\documentclass{amsart}

\usepackage[english]{babel}
\usepackage[utf8]{inputenc}
\usepackage{amsmath}
\usepackage{amsfonts}
\usepackage{amsthm}
\usepackage{amsrefs}
\usepackage{amssymb}
\usepackage{graphicx}
\usepackage{float}
\usepackage{mathrsfs}
\usepackage{calrsfs}
\usepackage{mathtools}
\usepackage{tabularx,array}
\usepackage{hyperref}

\theoremstyle{definition}
\newtheorem{definition}{Definition}[section]
\theoremstyle{plain}
\newtheorem{theorem}[definition]{Theorem}
\newtheorem{proposition}[definition]{Proposition}
\newtheorem{corollary}[definition]{Corollary}
\newtheorem{lemma}[definition]{Lemma}
\theoremstyle{remark}
\newtheorem{remark}[definition]{Remark}

\newcommand\xqed[1]{%
  \leavevmode\unskip\penalty9999 \hbox{}\nobreak\hfill
  \quad\hbox{#1}}
\newcommand\qet{\xqed{$\diamondsuit$}}

\newcommand{\Z}{\mathbb{Z}}
\newcommand{\R}{\mathbb{R}}
\newcommand{\G}{\mathbb{G}}
\newcommand{\GZ}{\G_{\Z}}

\newcommand{\mc}[1]{{\mathcal{#1}}}
\newcommand{\mr}[1]{{\mathrm{#1}}}

\DeclareMathOperator{\SCI}{SCI}
\DeclareMathOperator{\sgn}{sgn}
\DeclareMathOperator{\ind}{ind}
\DeclareMathOperator{\writhe}{w}
\DeclareMathOperator{\HN}{HN}
\DeclareMathOperator{\lk}{lk}

\DeclareMathOperator{\weight}{\omega}

\DeclareMathOperator{\Jp}{J^+}
\DeclareMathOperator{\Jm}{J^{--}}
\DeclareMathOperator{\St}{St}

\author{Piotr Suwara}
\address{Department of Mathematics, Massachusetts Institute of Technology, Cambridge, MA 02139}
\email{suwara@mit.edu}
\author{Albert Yue}
\address{Phillips Academy, Andover, MA 01810}
\email{albert.s.yue@gmail.com}
\title{An Index-Type Invariant of Knot Diagrams Giving Bounds for Unknotting Framed Unknots}
\date{\today}

\begin{document}

\maketitle

\begin{abstract}
    We introduce a~new knot diagram invariant called 
    the \emph{Self-Crossing Index} ($\SCI$).
    Using $\SCI$, we provide bounds for unknotting two families of framed unknots.
    For one of these families,
    unknotting using framed Reidemeister moves 
    is significantly harder than unknotting using regular Reidemeister moves.
    
    We also investigate the relation between $\SCI$ and Arnold's curve invariant $\St$,
    as well as the relation with Hass and Nowik's invariant,
    which generalizes cowrithe.
    In particular, the change of $\SCI$ under $\Omega3$ moves 
    depends only on the forward/backward character of the move,
    similar to how the change of $\St$ or cowrithe
    depends only on the positive/negative quality of the move.
\end{abstract}

\tableofcontents


\pagebreak

\section{Introduction}

Knots in $\R^3$ can be represented using planar diagrams
via taking a~generic projection onto a~plane $\R^2$
and marking each crossing with the information
about which strand is an~\emph{overcrossing} and which is an~\emph{undercrossing}.
This presentation is not unique, 
and two diagrams are equivalent if and only if they are connected
by a~sequence of Reidemeister moves of type $\Omega1, \Omega2$ and $\Omega3$
(see Figures \ref{fig:R1-cases}, \ref{fig:R2-cases}, \ref{fig:R3-cases}).

\begin{figure}[ht]
    \centering
    \includegraphics[width=0.8\textwidth]{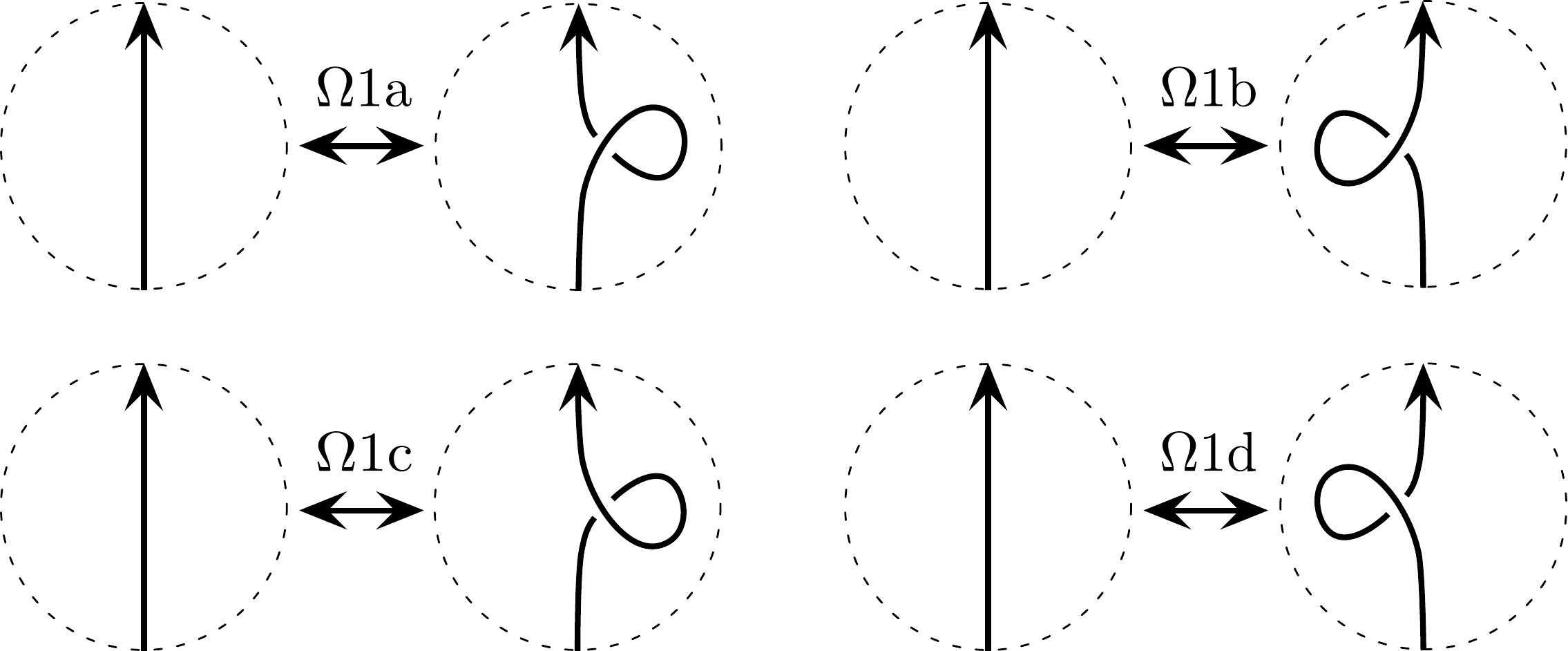}
    \caption{Right (a, c) and left (b, d) oriented Reidemeister moves of type I.}
    \label{fig:R1-cases}
\end{figure}

\begin{figure}[ht]
    \centering
    \includegraphics[width=0.8\textwidth]{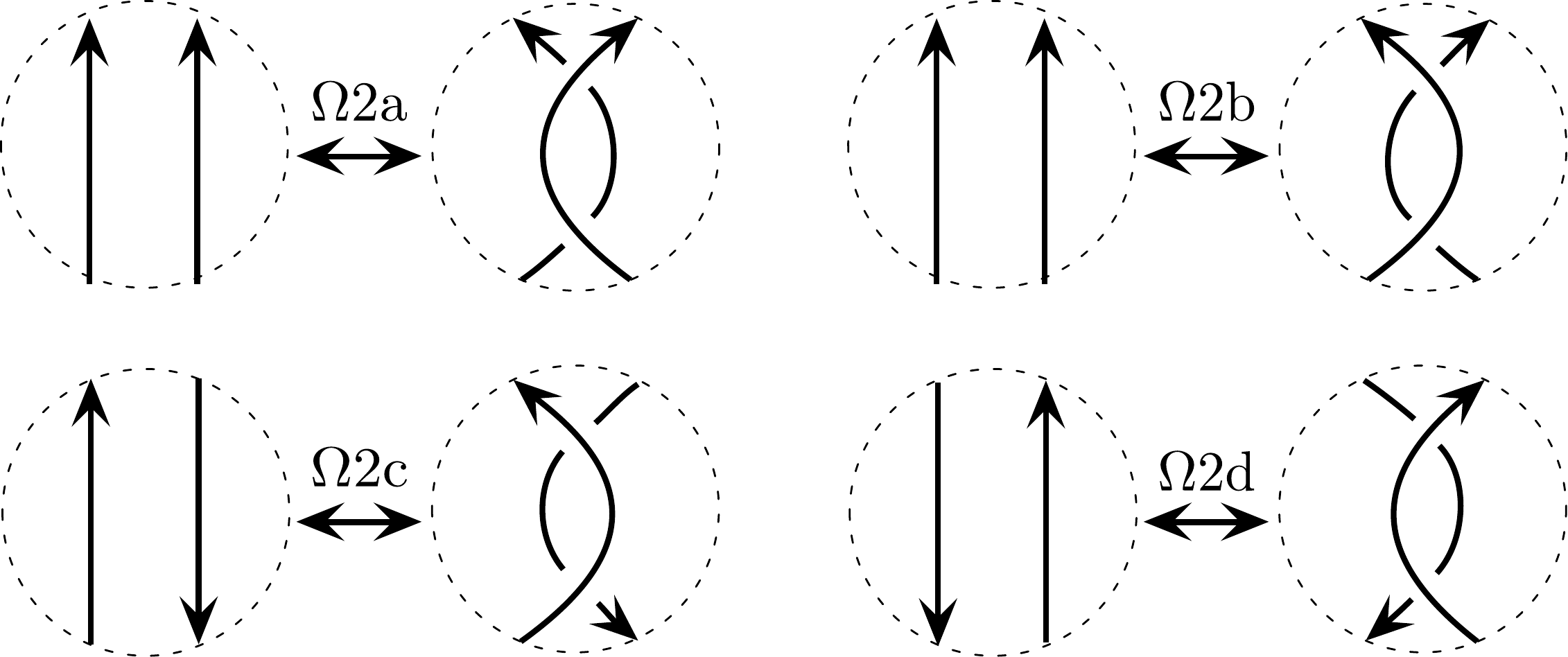}
    \caption{(a-b) Matched ($\Omega2m$) and (c-d) unmatched ($\Omega2u$) oriented Reidemeister moves of type II.}
    \label{fig:R2-cases}
\end{figure}

\begin{figure}[ht]
    \centering
    \includegraphics[width=0.8\textwidth]{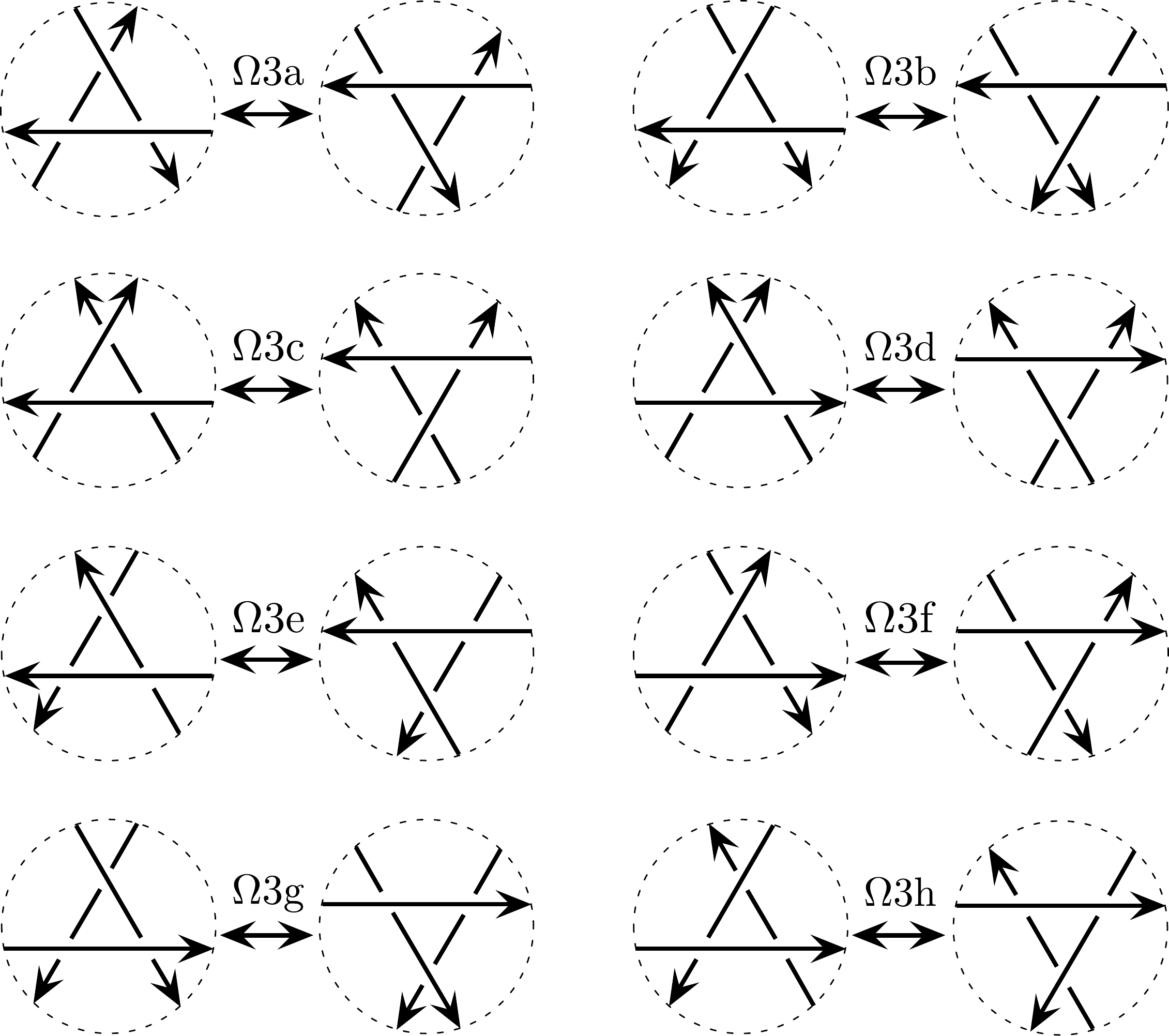}
    \caption{Oriented Reidemeister moves of type III.}
    \label{fig:R3-cases}
\end{figure}


An important problem in knot theory is the problem
of recognizing the unknot.
One way to approach it is through finding upper or lower bounds
for the length of a~minimal sequence of moves
required to untangle a~unknot diagram.
Recently, Lackenby \cite{L15} proved a~polynomial upper bound of $(236c)^{11}$
for unknotting, where $c$ is the number of crossings of a~diagram.
On the other hand, Hass and Nowik presented in \cite{HN10}
a~family of diagrams requiring $c^2/25$ moves to unknot,
using a~diagram invariant introduced in \cite{HN08}.
Another family of unknots with quadratic lower bound for unknotting
has been constructed by Hayashi, Hayashi, Sawada and Yamada \cite{HHSY}
using curve invariants defined by Arnold \cite{Arn94}.

In this paper, we construct a~new knot diagram invariant called $\SCI$
and prove that it provides bounds for unknotting \emph{framed knots}.
As with usual knots, two knot diagrams represent the same framed knot
if and only if they are connected by a~sequence of 
\emph{framed Reidemeister moves},
which include usual $\Omega2$ and $\Omega3$ moves,
but a~different kind of $\Omega1$ moves, which we call $\Omega1\mr{F}$
(see Figure \ref{fig:framed-Reidemeister-moves}).
\begin{figure}[ht]
    \centering
    \includegraphics[width=0.4\textwidth]{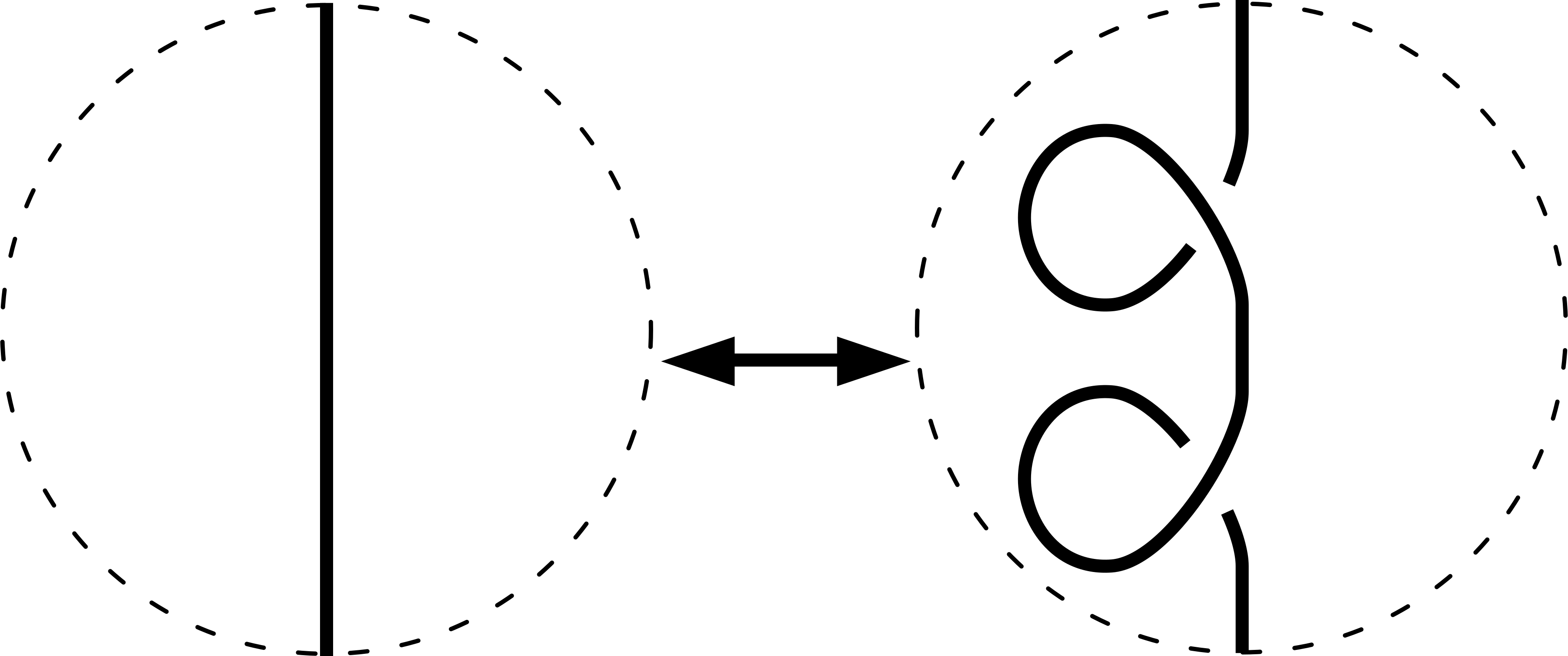}\\
    \quad \\
    \includegraphics[width=0.4\textwidth]{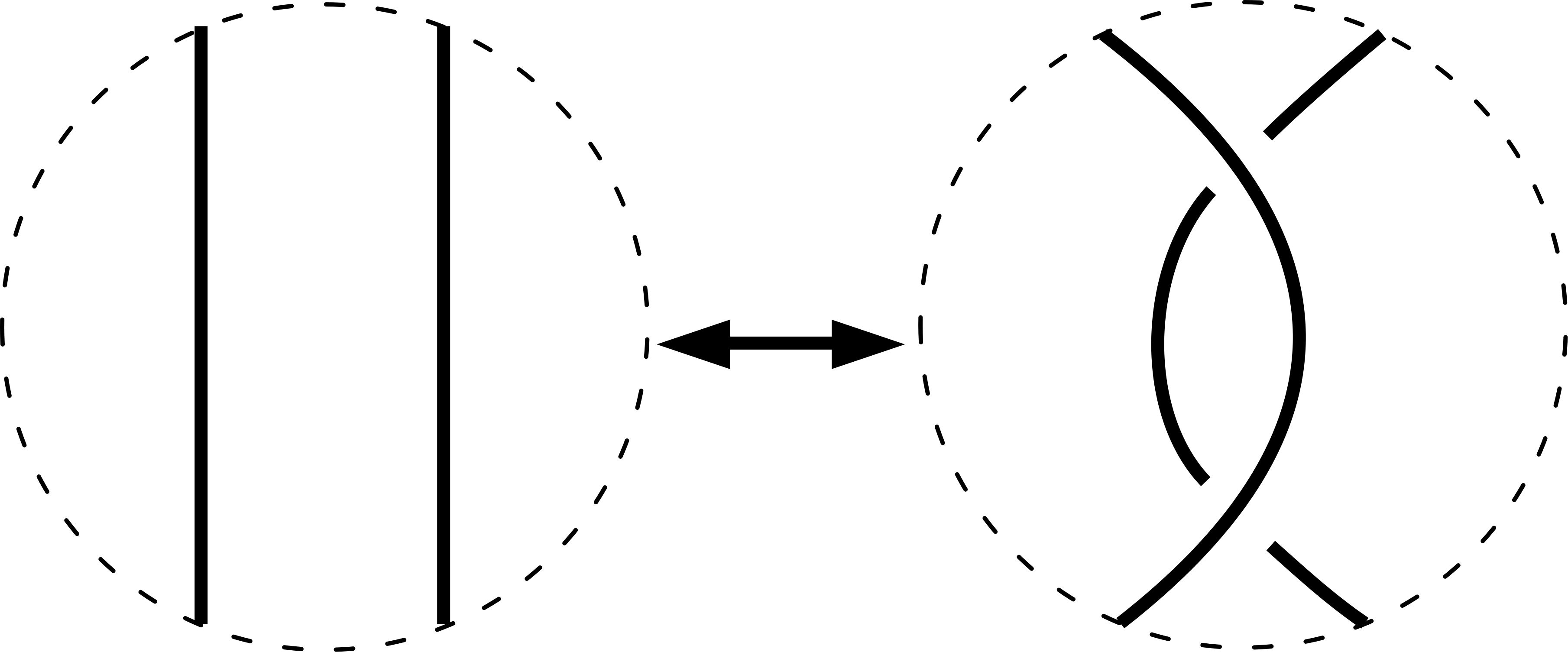}\\
    \quad \\
    \includegraphics[width=0.9\textwidth]{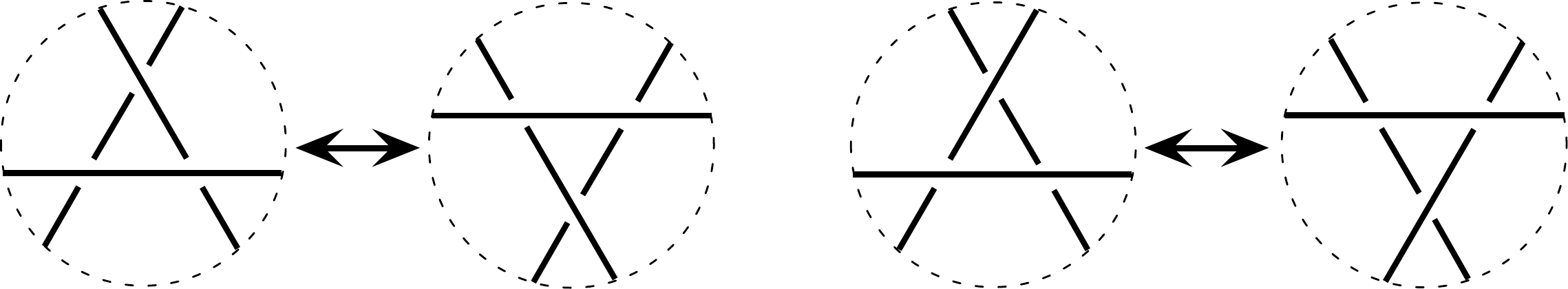}
    \caption{Unoriented Reidemeister moves of type $\Omega1\mr{F}, \Omega2$ and $\Omega3$ (top to bottom).}
    \label{fig:framed-Reidemeister-moves}
\end{figure}

The invariant $\SCI$ arises naturally as a~version
of invariants $\mathrm{CI}$ and $\mathrm{OCI}$ defined in \cite{S17},
which were used there to distinguish \emph{forward} and \emph{backward}
Reidemeister moves of type $\Omega3$.
\emph{Forward} moves of type $\Omega1$, $\Omega2$ and $\Omega3$
are presented in Figures \ref{fig:R1-cases}, \ref{fig:R2-cases} and \ref{fig:R3-cases}
as going from the diagram to the left to the diagram to the right.
As we will see, $\SCI$ distinguishes these, too.
The use of index closely resembles the technique used
by Vassiliev to define invariants of ornaments, 
i.e. sets of curves in a~plane \cite{V94}. 
Moreover,
Shumakovich \cite{Sh95} presented 
index-type formulas for Arnold's curve invariant $\St$,
while Viro \cite{Vi94} proved
formulas for Arnold's curve invariants $\Jp$ and $\Jm$.
The definition of $\SCI$ closely resembles a~formula for $\St$
given by Shumakovich, and $\SCI$ behaves in a~similar manner
under Reidemeister moves.

Our main result is the following: 
\begin{theorem} \label{thm:SCI-quadratic-bound}
    To unknot the family of diagrams $D_n$ (Figure \ref{fig:quadratic-family})
    as framed unknots one needs to use at least 
    $\frac{1}{2} \left( 3n^2 - n + 2 \right)$ moves of type $\Omega3$.
\end{theorem}
Since $\SCI$ is easy to compute, the above follows easily
if we understand how $\SCI$ changes under framed Reidemeister moves.
The following theorem is the main tool to obtain
bounds using $\SCI$.
\begin{theorem} \label{thm:SCI-changes}
    $\SCI$ increases by $1$ under forward $\Omega3$ moves
    and does not change under $\Omega1\mr{F}$ moves or $\Omega2$ moves.
\end{theorem}
Of course, unknotting a~framed unknot is not easier than unknotting
the same unknot using regular Reidemeister moves.
We show, using $\SCI$, that unknotting a framed unknot can be essentially harder:

\begin{theorem} \label{thm:framed-v-normal-Rmoves-comparison}
    The family of framed unknot diagrams $L_n$ is unknotted in $\Theta(n)$ moves
    using regular Reidemeister moves
    and in $\Theta(n^2)$ moves using framed Reidemeister moves.
\end{theorem}

The article is organized as follows.
In Section 2 we describe Arnold's invariants,
including Shumakovich's \cite{Sh95} and Viro's \cite{Vi94}
index-type formulas for the invariants.
In particular, we introduce indices of crossings
that will be used to define $\SCI$.
The main part of the article is Section 3, 
where we define the invariant $\SCI$.
We prove its additivity under connected sum and 
that it is a~Vassiliev diagram invariant of order $1$.
Then we prove Theorems 
\ref{thm:SCI-quadratic-bound},
\ref{thm:SCI-changes}
and 
\ref{thm:framed-v-normal-Rmoves-comparison}.
Finally, in Section 4,
we compare some of the properties of $\SCI$ 
to properties of the Hass-Nowik's invariant \cite{HN08} 
(denoted by $I_{\lk}$ in their paper).
In particular, 
we explain the relation between different types of $\Omega3$ moves:
positive/negative as defined by Arnold for curves \cite{Arn94},
ascending/descending as defined by \"Ostlund \cite{O01},
and forward/backward as defined by one of the authors in \cite{S17}.
The Appendix summarizes how known diagram invariants
change under different types of Reidemeister moves.

The authors want to thank the organizers of the MIT PRIMES program, 
especially Director Dr.~Slava Gerovitch,
Head Mentor Dr.~Tanya Khovanova,
and Chief Research Advisor Prof.~Pavel Etingof.
We are also grateful to Prof.~Maciej Borodzik,
who provided an important impulse for the research.

The project was supported by the Program for Research in Mathematics,
Engineering, and Science for High School Students (PRIMES) at MIT.

\section{Index-type description of Arnold's curve invariants}

In this section, we recall the definition of Arnold's curve invariants and state 
Shumakovich's \cite{Sh95} and Viro's \cite{Vi94} theorems describing these in terms of indices.
The definitions of indices will prove useful in the definition of the Self-Crossing Index, 
which is similar to the index-type description of the $\St$ curve invariant.

\subsection{Arnold's invariants}
When mentioning Reidemeister moves on curves we consider
the moves obtained from regular Reidemeister moves
by forgetting the information about over- and undercrossings.
The distinction between $\Omega1$, matched $\Omega2$, 
unmatched $\Omega2$, and $\Omega3$ moves 
carries over to the case of curves, 
as well as the notions of left and right $\Omega1$ moves for oriented curves
(see Figures \ref{fig:R1-cases}, \ref{fig:R2-cases} and \ref{fig:R3-cases};
while we need to choose an orientation to distinguish
between matched and unmatched $\Omega2$ moves,
the matched/unmatched type
does not depend on the orientation chosen).

By \emph{positive} (or \emph{forward}) moves of type $\Omega1$ or $\Omega2$ 
we define moves that create new crossings; 
their converses are called \emph{negative} (or \emph{backward}).
In order to define Arnold's invariant we also 
need to define what \emph{positive} and \emph{negative}
moves of type $\Omega3$ are.

\begin{definition}[vanishing triangle]
	The \emph{vanishing triangle} 
    of a~$\Omega3$ move is the triangle formed by the three edges
    contained in the diagram of a~$\Omega3$ move (see \ref{fig:R3-cases})
    which ends are the three crossings involved in a~$\Omega3$ move.
\end{definition}

\begin{definition} [positive and negative $\Omega3$ move]
    \label{def:R3-positive}
    Consider an $\Omega3$ move performed on a~closed oriented curve $C$.
    Consider the vanishing triangle of this move.
    Assign an orientation to the vanishing triangle 
    corresponding to the order in which its sides appear
    if we move along $C$ beginning at an arbitrary point. 
    Let $n$ be the number of sides of the vanishing triangle whose orientation agrees
    with the orientation of the triangle and let $q = (-1)^n$. 
    Then a $\Omega3$ move is considered positive if it changes $q$ from $-1$ to $+1$ 
    and negative if the reverse occurs.
    \qet
\end{definition}

\begin{remark}
    The definitions of positive and negative moves carry over
    to regular Reidemeister moves (i.e. on knot/link diagrams).
    For moves of type $\Omega1$ and $\Omega2$
    positive (resp. negative) moves are the same as 
    forward (resp. backward) moves.
    For moves of type $\Omega3$, these notions are different,
    and the relationship between these is clarified in Subsection \ref{subs:type-3-moves}.
\end{remark}

\begin{definition} [Arnold \cite{Arn94}] \label{def:arnold-invariants}
    The Arnold invariants $\Jp$, $\Jm$, and $\St$ are defined by the following rules:\vspace{-5pt}
    \begin{enumerate}
        \item Orientation of the curve does not affect the invariants. \vspace{-5pt}
        \item $\Jp$ changes by $+2$ under positive matched $\Omega2$ moves, 
            and remains unchanged under unmatched $\Omega2$ moves and $\Omega3$ moves. \vspace{-5pt}
        \item $\Jm$ changes by $+2$ under positive unmatched $\Omega2$ moves, 
            and remains unchanged under matched $\Omega2$ moves and $\Omega3$ moves. \vspace{-5pt}
        \item $\St$ changes by $+1$ under positive $\Omega3$ moves, 
            and remains unchanged under $\Omega2$ moves. \vspace{-5pt}
        \item For curves $K_0$ and $K_i$, for $i \in \mathbb{N}_0$ (see Figure \ref{fig:arnold-bases}), 
            \vspace{-5pt}
            \begin{enumerate}
                \item $\Jp(K_{i+1}) = -2i$, $\Jp(K_0) = 0$; 
                \item $\Jm(K_{i+1}) = -3i$, $\Jm(K_0) = -1$; 
                \item $\St(K_{i+1}) = i$, $\St(K_0) = 0$.
                    \qet
            \end{enumerate}
    \end{enumerate}
    \begin{figure}[ht]
        \centering
        \includegraphics[scale=0.3]{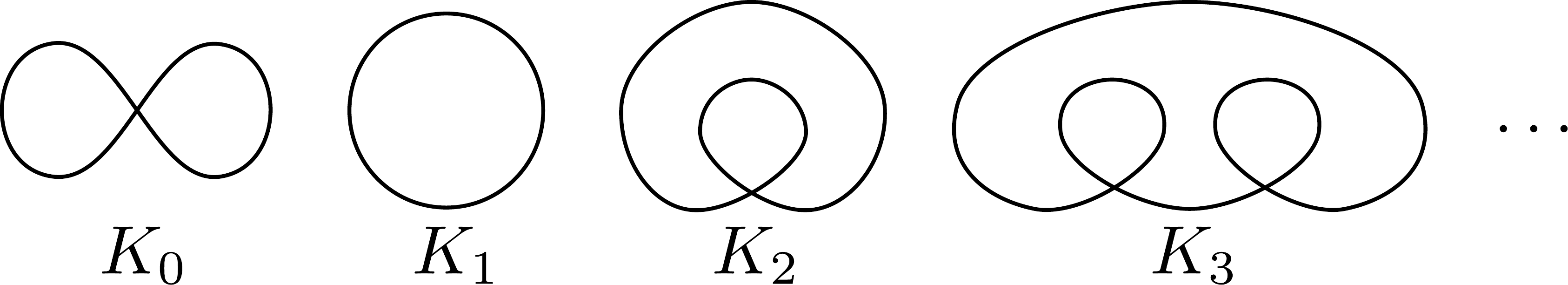}
        \caption{The base cases $K_i$ for which the Arnold invariants are defined.}
        \label{fig:arnold-bases}
    \end{figure}
\end{definition}

Arnold proved that such invariants exists, and their uniqueness follows from the fact
that any curve may be obtained from one of the $K_i$'s using $\Omega2$ and $\Omega3$ moves.

Note that these also can be used to obtain bounds for unknotting,
as did Hayashi, Hayashi, Sawada and Yamada \cite{HHSY}.

\subsection{Indices with respect to a~curve}
We now proceed to define indices of points in the plane with respect to a~given curve $C$.

\begin{definition}[the index of a point with respect to a curve]
    Let $\gamma:S^1 \to \R^2$ represent an~oriented curve $C$, 
    and let $p \in \R^2 \setminus \gamma(S^1)$.
    Then we define \emph{the index of $p$ with respect to $C$},
    denoted as $\ind_C(p)$, 
    to be the degree of the map $\tilde \gamma_p:S^1 \to S^1$
    defined by
    \[\tilde \gamma_p(t) = \frac{\gamma(t) - p}{\| \gamma(t) - p \|}. \]

    We will drop the subscript $C$ from notation and write
    $\ind(p)$ whenever it causes no confusion.
    \qet
\end{definition}

Since the index with respect to $C$ is equal for all points in a~connected component
of the complement of $C$, we can define the following:
\begin{definition}[indices of regions, edges and crossings]
    Let $C$ be an~oriented curve.

    Let $r$ be a~region of $\R^2$, i.e. a~connected component
    of $\R^2 \setminus C$.
    Then we define the \emph{index of the region $r$ with respect to $C$} to be
    \[ \ind_C(r) = \ind_C(p) \]
    for any $p \in r$.
    We denote the set of all regions by $\mc{R}(C)$.

    Denote by $\mc{C}(C)$ the set of all crossings of $C$,
    and by $\mc{E}(C)$ the set of all edges of $C$,
    i.e. connected components of $C \setminus \mc{C}(C)$.
    Let $e$ be an~edge of $C$ and define its \emph{index with respect to $C$} to be
    \[ \ind_C(e) = \frac 1 2  \sum_{r \in \mc{R}(e)} \ind_C(r) \]
    where $\mc{R}(e) \subset \mc{R}(C)$ is the set of two regions adjacent to $e$.

    Let $c$ be a~crossing of $C$, define its \emph{index with respect to $C$} to be
    \[ \ind_C(c) = \frac 1 4 \sum_{r \in \mc{R}(c)} \ind_C(r) \]
    where $\mc{R}(c) \subset \mc{R}(C)$ is the set of four regions adjacent to $c$
    (counted with multiplicity).
    \qet
\end{definition}

\subsection{Viro's formulas for $\Jp$ and $\Jm$}
To introduce Viro's formulas for $\Jp$ and $\Jm$
we recall the definition of smoothing of a~crossing:
\begin{definition}
    \label{def:smoothing}
    Let $c$ be a crossing of an~oriented curve $C$. 
    Then the \emph{smoothing} of $C$ 
    consists of two (potentially intersecting) curves 
    created by removing the crossing $c$ and replacing it with two non-intersecting strands 
    that preserve the original orientation 
    (cf. with Figure \ref{fig:smoothing}).
    \qet
\end{definition} 
\begin{figure}[ht] 
    \centering
    \includegraphics[width=0.7\textwidth]{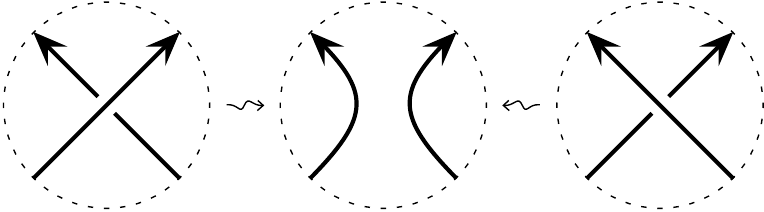}
    \caption{Smoothing a~crossing of a~knot
        is just smoothing the crossing of the underlying curve
        and forgetting about the crossing information.}
    \label{fig:smoothing}
\end{figure}

Note that if $C$ is not oriented,
taking any of the two possible orientations gives the same smoothing.
The definition above goes through for crossings between oriented curves
as well as for crossings of oriented link diagrams.

Now we are ready to state
\begin{theorem}[Viro \cite{Vi94}] \label{thm:explicit-Jpm}
    Let $C$ be a curve with $n$ double points. 
    Let $\bar{C}$ be the diagram obtained by smoothing all crossings of $C$, 
    and let $\mathcal{R}(\bar{C})$ be the set of regions in the complement of $\bar{C}$. 
    Then
    \[ \Jp(C) = 1 + n - \sum_{r \in \mathcal{R}(\bar{C})} \left( \chi(r) \ind^2(r) \right) ,\]
    \[ \Jm(C) = 1 - \sum_{r \in \mathcal{R}(\bar{C})} \left( \chi(r) \ind^2(r) \right) ,\]
    where $\chi$ is the Euler characteristic and $n$ the number of crossings of $C$.
\end{theorem}
We see that $\Jp$ and $\Jm$ admit explicit descriptions using indices.
This is useful for both calculating the values of $\Jp$ and $\Jm$
for given diagrams, as well as proving properties of $\Jp$ and $\Jm$,
such as:
\begin{proposition}
    \label{thm:Jpm-R1-change}
    Choose an orientation of $C$. 
    Under a~left (resp. right) positive $\Omega1$ move, 
    $\Jm$ changes by $-2\ind(c)-1$ (resp. $2\ind(c)-1$) 
    and $\Jp$ changes by $-2\ind(c)$ (resp. $2\ind(c)$),
    where $c$ is the crossing created by the $\Omega1$ move.
    \begin{proof}
        The addition of a loop by a~left positive $\Omega1$ move affects formulas 
        for $\Jp$ and $\Jm$ in only two ways:  
        adding a new region to the complement of $\bar{C}$
        and changing the Euler characteristic of the region that the loop is made in.
        Let $c$ be the crossing created by the $\Omega1$ move. 
        Then the new region in the complement of $\bar{C}$ is a~disk with Euler characteristic $1$
        and index $\ind(c)+1$. 
        The Euler characteristic of the region surrounding the loop decreases by $1$
        and the index remains $\ind(c)$. 
        Thus, the change to $\Jm$ is
        \[ \Delta\Jm  = -(\ind(c)+1)^2 - ((\chi_1 - 1)\ind^2(c) - \chi_1 \ind^2(c)) = -2\ind(c) - 1 \]
        under one positive $\Omega1$ move,
        where $\chi_1$ is the original Euler characteristic of the region surrounding the loop.

        From $\Jp = \Jm + n$, where $n$ is the number of crossings, 
        it follows that $\Jp$ changes by $-2\ind(c)$ under a~left positive $\Omega1$ move.
        
        Changing the orientation of $C$, 
        we obtain the desired results for right positive $\Omega1$ moves.
    \end{proof}
\end{proposition}

\subsection{Shumakovich's formulas for $\St$}
We proceed to Shumakovich's formulas for $\St$.
First we need to define \emph{weights}.

Fix an arbitrary point $p$ on the oriented curve $C$ which is not one of its $n$ crossings.
Label the edges from $1$ to $2n$ following the orientation of the curve,
with the edge containing $p$ being labeled by $1$.

\begin{definition}[weight] \label{def:weight}
    Consider a~crossing $c$.
    Denote the edges pointing towards $c$ by $e_i$ and $e_j$, 
    where $i$ and $j$ are their respective labels,
    with $e_i$ crossing $e_j$ from left to right
    (see Figure \ref{fig:St-weight}).
    Let $\sgn(k)$ be the sign of the integer $k$. Then set
    \[ \weight(c) = \sgn(i-j) ,\]
    \[ \weight(e_i) = \sgn(i-j) ,\]
    \[ \weight(e_j) = -\sgn(i-j) .\]

    Let $r_W$ be the region directly to the left of $e_i$, $r_E$ be the region directly 
    to the right of $e_j$, $r_S$ be the region directly to the right of $e_i$ and left of $e_j$, 
    and $r_N$ be the remaining region surrounding $c$ (see Figure \ref{fig:St-weight}). 
    The weight $\weight(r)$ of a region is the sum of the contributions of all adjacent crossings
    (with multiplicity two if a~region is adjacent to a~crossing in two ways),
    denoted $\weight_c(r)$, which are equal to
    \[ \weight_c(r_W) = \weight_c(r_E) = \frac{1}{2} \sgn(i-j) ,\]
    \[ \weight_c(r_N) = \weight_c(r_S) = -\frac{1}{2} \sgn(i-j) .\]
    \qet
\end{definition}

\begin{remark}
    Writhe of a~knot diagram $D$ is the sum of signs of all crossings:
    $$w(D) = \sum_{c \in \mc{C}(D)} \sgn(c).$$
    Using the weights defined above, one can try to define
    a~curve invariant via $\sum_{c \in \mc{C}(C)} \weight(c)$.
    This is not invariant under the choice of the point $p$,
    but an easy inspection shows that if we subtract $2 \ind(p)$,
    one obtains a~curve invariant.
    Checking how it changes under Reidemeister moves
    and calculating the value on a~simple closed curve one obtains the winding number of $C$:
    $$\mathrm{wind(C)} = -2 \ind(p) + \sum_{c \in \mc{C}(C)} \weight(c).$$
\end{remark}

\begin{figure}[ht]
    \centering
    \includegraphics[width=0.25\textwidth]{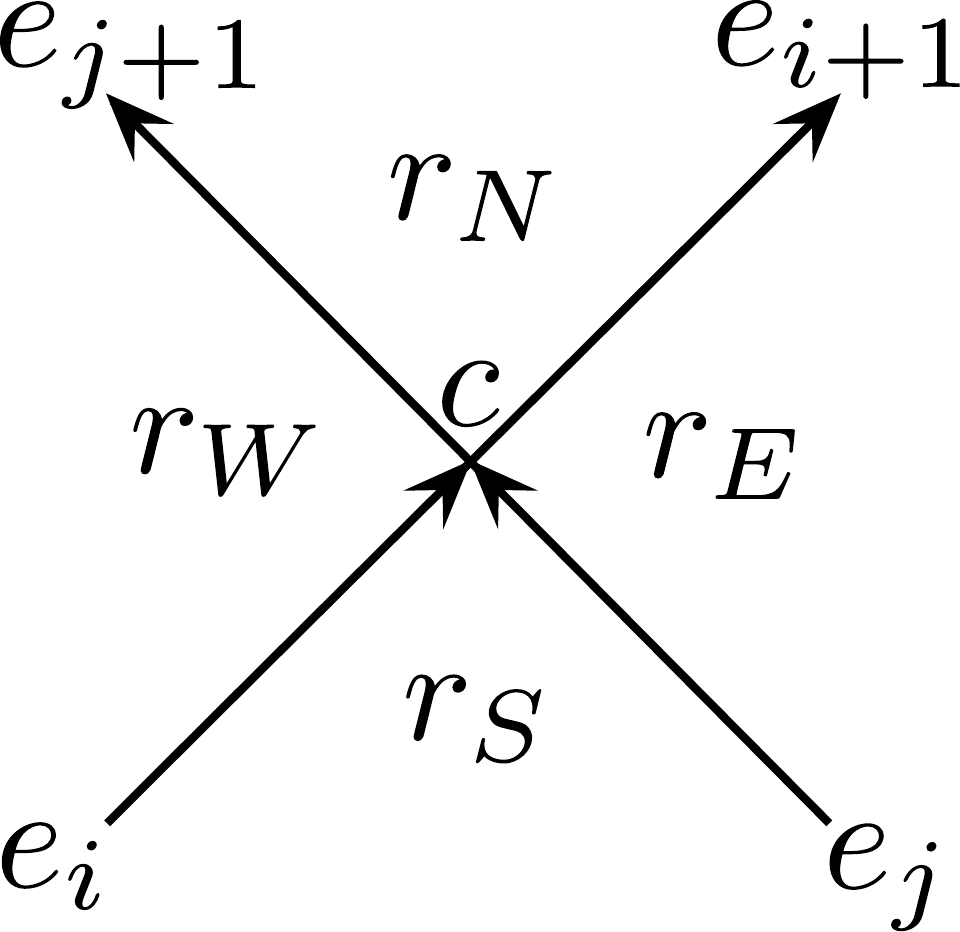}
    \caption{Calculation of weight for crossing $c$ and the edges $e_i$ and $e_j$ pointing towards $c$, and the contribution to the weight of the surrounding regions.}
    \label{fig:St-weight}
\end{figure}

\begin{theorem}[Shumakovich \cite{Sh95}] \label{thm:explicit-St}
    Let $C$ be an oriented curve. 
    Then 
    \begin{equation}
        \St(C) = \sum_{c \in \mathcal{C}(C)} 
        \left( \weight(c) \ind(c) \right) + \delta^2 - \frac{1}{4} ,
        \label{eqn:Shu-1}
    \end{equation}
    \begin{equation}
        \St(C) = \frac{1}{2} \sum_{e \in \mathcal{E}(C)} 
        \left( \weight(e) \ind^2(e) \right) + \delta^2 - \frac{1}{4} ,
        \label{eqn:Shu-2}
    \end{equation}
    \begin{equation}
        \St(C) = \frac{1}{3} \sum_{r \in \mathcal{R}(C)} 
        \left( \weight(r) \ind^3(r) \right) + \delta^2 - \frac{1}{4} ,
        \label{eqn:Shu-3}
    \end{equation}
    where $\delta = \ind(e_p)$, where $e_p$ is is the edge containing $p$.
\end{theorem}

Again, these descriptions allow to easily examine some properties of $\St$.

\begin{proposition}
    \label{thm:St-R1-change}
    Arnold's invariant $\St$ changes by $+\ind(c)$ under 
    a~left positive $\Omega1$ move and $-\ind(c)$ under a~right positive $\Omega1$ move,
    where $c$ is the new crossing formed by the $\Omega1$ move.
    \begin{proof}
        This follows from Equation (\ref{eqn:Shu-1}).
        First, choose a~point $p$ which does not lie 
        on the edge the $\Omega1$ move is applied to.
        We can do so if the diagram is not a~trivial unknot diagram,
        in which case the proposition is easily checked to be true.

        A~positive $\Omega1$ move adds a~crossing.
        If we keep the same starting point $p$,
        the numbering of edges changes,
        but weights of any other crossings stay the same,
        since the labels of adjacent edges all shift by either $0$ or $2$.
        Let the three edges connected to the new crossing be numbered $k$, $k+1$, and $k+2$.
        For a~left $\Omega1$ move, we get $\weight(c) = \sgn((k+1)-k) = +1$,
        and for a~right one we get $\weight(c) = \sgn(k-(k+1)) = -1$,
        proving the proposition.
    \end{proof}
\end{proposition}


\section{The Self-Crossing Index and bounds for unknotting framed knots}

In this section, we introduce a new knot diagram invariant, 
called the \emph{Self-Crossing Index}, or $\SCI$. 
We prove that it is additive under connected sum
and that it is Vassiliev of order $1$.
We finally show how it provides bounds
for unknotting framed knots
via Theorems \ref{thm:SCI-quadratic-bound}, \ref{thm:SCI-changes}
and \ref{thm:framed-v-normal-Rmoves-comparison}.

\subsection{Definition and properties of $\SCI$}
In the previous section, we defined indices of points with respect to a~curve,
weights of regions, edges and crossings of a~closed curve,
as well as smoothing of a~crossing.
All these generalize to the case of a~knot diagram
by considering the underlying curve of a~diagram.

\begin{definition}[Self-Crossing Index] \label{def:SCI}
    Let $D$ be an oriented knot diagram 
    and let $\mathcal{C}(D)$ be the set of crossings of $D$. Then
    \[ \SCI(D) = \sum_{c \in \mathcal{C}(D)} \sgn(c) \ind(c) \]
    where $\sgn(c)$ is the \emph{sign} of the crossing $c$.
    \qet
\end{definition}

We immediately notice similarity with Equation (\ref{eqn:Shu-1}).
Suppose the knot diagram $D$ is ascending.
Let $p$ be a~\emph{lowest point} of the diagram $D$, 
that is a~point such that
if we move along $D$ starting at $p$,
then each crossing is passed through its undercrossing first.
If the same point $p$ is taken to calculate weights as in the previous section,
then we obtain that $\weight(c) = \sgn(c)$ for any crossing $c$.
Thus
\[ \St(C) = \sum_{c \in \mathcal{C}(D)} \left( \sgn(c) \ind(c) \right) + \delta^2 - \frac{1}{4}\]
and therefore
\begin{theorem}
    \label{thm:SCI-St}
    For an ascending knot diagram $D$ and its underlying curve $C$, 
    \[ \SCI(D) = \St(C) - \delta^2 + \frac{1}{4} ,\]
    where $\delta = \ind(p)$, $p$ being a lowest point of $D$.
\end{theorem}

From this we immediately obtain formulas for $\SCI$ similar to
Equations (\ref{eqn:Shu-2}) and (\ref{eqn:Shu-3}),
under the assumption that $D$ is ascending.
However, we aim to prove such formulas for $\SCI$ in full generality.
For this, we need to define modified weights depending
on the signs of the crossings rather than on the topology of the curve.

\begin{definition} \label{def:SCI-weight}
    Using a set-up similar to that used to define weight (Definition \ref{def:weight}), 
    we define, for a crossing $c$, 
    \[ \tilde{\weight}(e_i) = \sgn(c) ,\]
    \[ \tilde{\weight}(e_j) = -\sgn(c) ,\]
    \[ \tilde{\weight}_c(r_W) = \tilde{\weight}_c(r_E) = \frac{1}{2} \sum_{c \in \mathcal{C}(r)} \sgn(c) ,\]
    \[ \tilde{\weight}_c(r_N) = \tilde{\weight}_c(r_S) = -\frac{1}{2} \sum_{c \in \mathcal{C}(r)} \sgn(c) ,\]
    \[ \tilde{\weight}(r) = \sum_{c \in \mathcal{C}(r)} \tilde{\weight}_c(r), \]
    where $\mathcal{C}(r)$ is the set of all crossings adjacent to the region
    (again, counted with multiplicities).
    \qet
\end{definition}

\begin{theorem} \label{def:SCI-alt}
    Let $D$ be an oriented knot diagram and let $\mathcal{E}(D)$ and $\mathcal{R}(D)$ be the set of edges and regions of $D$, respectively. Then
    \begin{equation}
        \SCI(D) = \frac{1}{2} \sum_{e \in \mathcal{E}(D)} \tilde{\weight}(e) \ind^2(e) ,
        \label{eqn:SCI-2}
    \end{equation}
    \begin{equation}
        \SCI(D) = \frac{1}{3} \sum_{r \in \mathcal{R}(D)} \tilde{\weight}(r) \ind^3(r) .
        \label{eqn:SCI-3}
    \end{equation}
    \begin{proof}
        We follow the argument given in \cite{Sh95}.

        To show that these three formulas of $\SCI$ are equivalent, 
        we will show that the calculations are equivalent in a~neighborhood of a~crossing $c$. 
        From the definition, the contribution of a~crossing $c$ to $\SCI$ is $\sgn(c)\ind(c)$. 

        Denote $\alpha = \ind(c)$. 
        Then $\ind(e_i) = \alpha + \frac{1}{2}$, $\ind(e_j) = \alpha - \frac{1}{2}$.
        Therefore the contribution of the edges $e_i, e_j$
        to the sum (\ref{eqn:SCI-2}) is
        \[ 
            \frac{1}{2} \left( \sgn(c) \left( \alpha + \frac{1}{2} \right)^2 
            - \sgn(c) \left( \alpha - \frac{1}{2} \right)^2 \right)
        = \alpha \sgn(c) = \sgn(c) \ind(c).\]
        This proves the identity (\ref{eqn:SCI-2}).

        Similarly, we can consider the regions around $c$, 
        and consider the contribution to $\tilde{\weight}(r)$ that $c$ makes (i.e. $\pm \sgn(c)/2$).
        Since $\tilde{\weight}(r)$ is just the sum of contributions of adjacent crossings,
        we can rewrite (\ref{eqn:SCI-3}) as
        \[ \SCI(D) = \frac{1}{3} \sum_{c \in \mathcal{C}(D)} 
        \sum_{r \in \mathcal{R}(c)} \pm \sgn(c) \ind^3(r)/2  ,\]
        where $\mathcal{C}(D)$ is the set of crossings of $D$,
        $\mathcal{R}(c)$ is the set of four regions surrounding a crossing $c$
        and the sign $\pm$ depends on whether $r$ is $r_E, r_W, r_N$ or $r_S$ for that crossing.
        However, since
        $\ind(r_W) = \alpha + 1$, $\ind(r_E) = \alpha -1$, and $\ind(r_N) = \ind(r_S) = \alpha$,
        we have
        \begin{align*}
            \frac 1 3 \sum_{r \in \mathcal{R}(c)}  \pm \sgn(c) \ind^3(r)/2 
            &= 
            \frac{1}{3} \left( \frac{\sgn(c)}{2} \left( \left( \alpha + 1 \right)^3 
            + \left( \alpha - 1 \right)^3 \right) 
            - \frac{\sgn(c)}{2} \left( \alpha^3 + \alpha^3 \right) \right)
            \\
            &= \alpha \sgn(c),
        \end{align*}
        which proves (\ref{eqn:SCI-3}).
    \end{proof}
\end{theorem}

Another remarkable property of $\SCI$ is its additivity under connected sums.

\begin{theorem} \label{thm:SCI-connected-sum}
    Let $D$ and $E$ be two knot diagrams and $D \# E$ denote their connected sum.
    Then
    \begin{equation}
        \SCI(D \# E) = \SCI(D) + \SCI(E).
        \label{eqn:SCI-additivity}
    \end{equation}
    \begin{proof} 
        The connected sum of $D$ and $E$ leaves the signs of the crossings unchanged.
        In addition, the indices of the regions do not change, 
        as the operation simply merges together two regions with the same index.
        Thus the indices and signs of the crossings stay the same
        and summing along all the crossings of $D$ and $E$
        gives the identity in the proposition.
    \end{proof}
\end{theorem}

We also note that $\SCI$ is a~\emph{Vassiliev (diagram) invariant} of order $1$.

\begin{definition}[finite type/Vassiliev diagram invariant \cite{CDM12}] \label{def:vassiliev}
    Let $D$ be a~knot diagram.
    Let $S$ be a~subset of crossings of $D$, $S \subset \mc{C}(D)$.
    For a~knot diagram invariant $I$ we define, inductively,
    \[ I_{S}(D) = I_{\left( S \setminus \{c\} \right)}(D) - I_{\left( S \setminus \{c\} \right)}(D_{c}) ,\]
    where $c$ is an arbitrary crossings in $S$ 
    and $D_{c}$ is the diagram $D$ with the crossing $c$ changed.
    Equivalently,
    \[ I_{S}(D) = \sum_{X \subseteq S} (-1)^{|X|} I(D_{X}) ,\]
    where $D_X$ is the diagram $D$ with all the crossings from $X$ changed.

    We define $I$ to be a~\emph{Vassiliev invariant of order at most} $m \geq 0$ 
    (or \emph{finite type invariant})
    if $I_S(D) = 0$ for any diagram $D$ and any subset $S \subset \mc{C}(D)$ such that $|S|=m+1$.
    We say that $I$ is exactly \emph{of order $m$} if it is of order at most $m$
    and there is a~diagram $D$ and set $S \subset \mc{C}(D)$ such that $|S|=m$
    and $I_S(D) \neq 0$.
    \qet
\end{definition}

\begin{remark}
    Diagram invariants arising from curve invariants (e.g. $\St$, $J^+$, $J^-$) 
    are Vassiliev of order $0$.
\end{remark}

\begin{theorem} \label{thm:SCI-finite-type}
    $\SCI$ is a Vassiliev invariant of order 1.
    \begin{proof}
        Take any knot diagram $D$ with at least two crossings. 
        Let $a \neq b$ be two crossings of $D$.
        Clearly, changing a~crossing does not change any indices of crossings
        and it changes the sign of one crossing.
        Thus we have
        \begin{equation*}\label{eq:SCI-finite-type-0}
            \begin{split}
                \SCI_{\{a\}}(D) & = \SCI(D) - \SCI(D_a) \\
                & = 2 \sgn(a) \ind(a) . 
            \end{split}
        \end{equation*}
        But similarly we have $\SCI_{\{a\}}(D_b) = 2 \sgn(a) \ind(a)$.
        Thus $\SCI_{\{a,b\}}(D) = 0$, so $\SCI$ is Vassiliev of order at most $1$.

        Finally, taking any diagram $D$ which has a~crossing $a$
        of index $\ind(a) \neq 0$ we get that $\SCI$ is of order $1$.
    \end{proof}
\end{theorem}

\subsection{Bounds for unknotting via $\SCI$}
We now prove Theorem \ref{thm:SCI-changes}, 
which is the key tool in establishing bounds for unknotting framed knots
using $\SCI$.

\begin{proof}[Proof of Theorem \ref{thm:SCI-changes}]
    For $\Omega1\mr{F}$ or $\Omega2$ moves, 
    each move creates or removes two crossings of opposite sign
    and of the same index, thus preserving $\SCI$.

    For $\Omega3$ moves, 
    the signs of the crossings remain unchanged, 
    so it is enough to consider changes of indices of the crossings.
    There are eight cases that need to be considered (cf. Figure \ref{fig:R3-cases}).
    We consider the case of an $\Omega3a$ move
    since all the other cases are similar.

    \begin{figure}[ht] 
        \centering
        \includegraphics[width=0.6\textwidth]{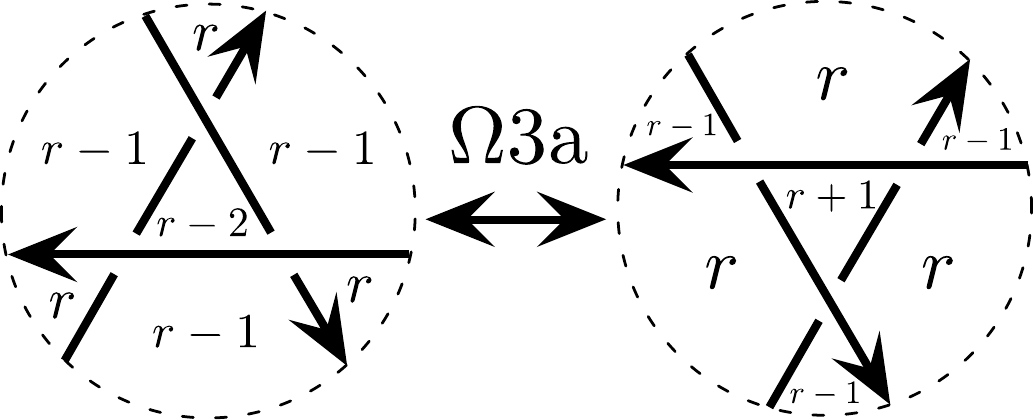}
        \caption{Changes to indices of regions adjacent 
        to the crossings involved in a $\Omega3a$ move.}
        \label{fig:R3aSCI}
    \end{figure}

    For a~forward $\Omega3a$ move (left to right in Figure \ref{fig:R3aSCI}), 
    all three indices of crossings increase by $1$.
    The signs of the crossings are $+1, +1$ and $-1$,
    so the overall change to $\SCI$ equals $+1$, as claimed.
\end{proof}

\begin{figure}[ht]
    \centering
    \includegraphics[width=0.8\textwidth]{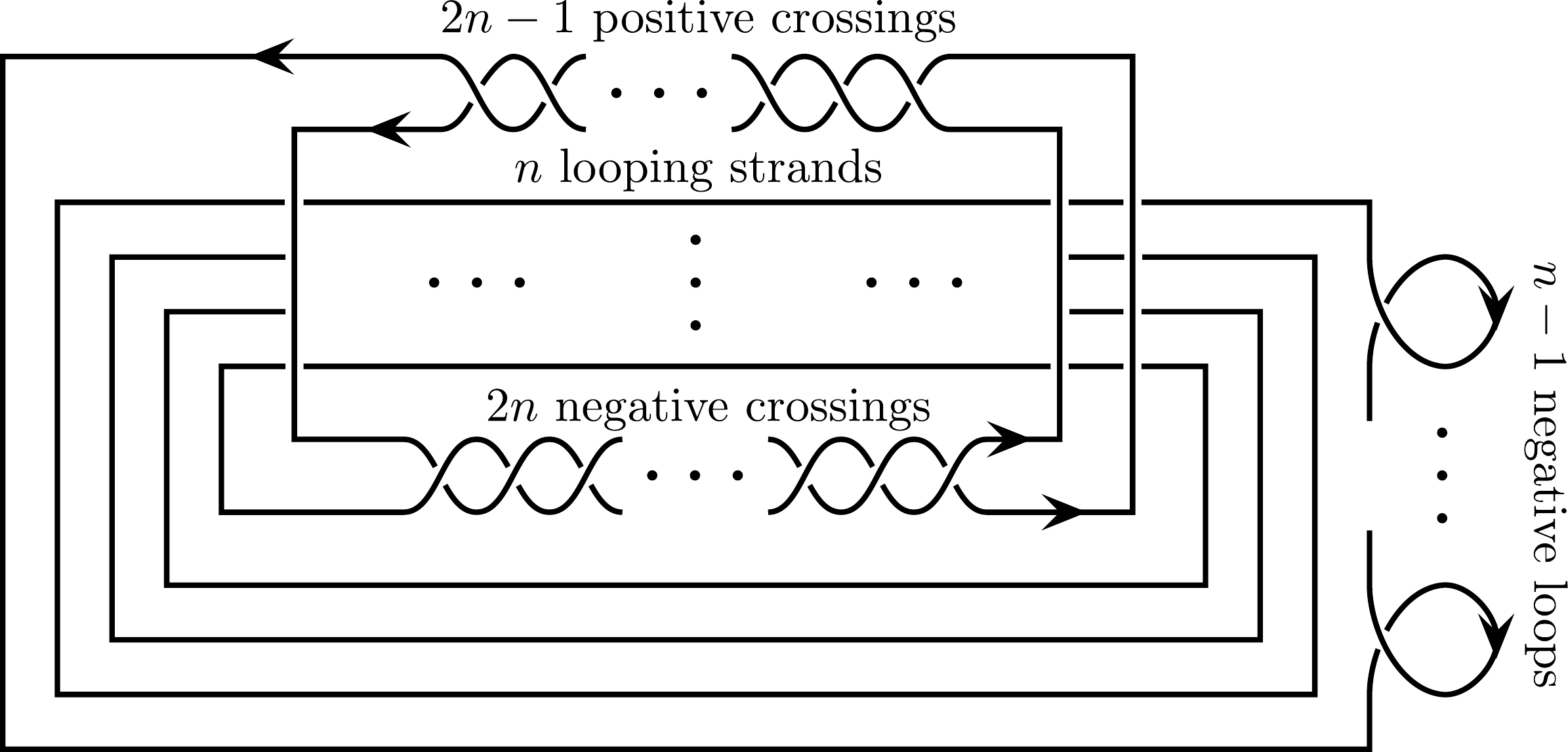}
    \caption{The family of unknots $D_n$.}
    \label{fig:quadratic-family}
\end{figure}

\begin{corollary}
    For an~unknot diagram $D$,
    the number of Reidemeister moves of type $\Omega3$ needed to unknot $D$
    is greater or equal to $|\SCI(D)|$.
    \label{cor:SCI-bounds}
    \begin{proof}
         Since $\SCI$ is zero for any trivial knot diagram,
         the result follows from Theorem \ref{thm:SCI-changes}.
    \end{proof}
\end{corollary}

\begin{proof}[Proof of Theorem \ref{thm:SCI-quadratic-bound}]
    Through computation, we find that 
    \[ \SCI(D_n) = \frac{1}{2} \left( 3n^2 - n + 2 \right) .\]
    The desired result follows from Corollary \ref{cor:SCI-bounds}.
\end{proof}

Hass and Nowik \cite{HN10} found a lower bound of $2n^2 + 3n -2$ using $\HN$ 
(see Definition \ref{def:Hass-Nowik}) for unknotting a~similar family of unknots 
using regular Reidemeister moves.
Using the same procedure as Hass and Nowik \cite{HN10}, 
one obtains a~quadratic lower bound of $2n^2 + 2n - 1$ for unknotting $D_n$,
even in the framed setting.
Indeed, we have
\[ \HN(D_n) = n X_n + n X_{-n} + (2n-1)X_{-1} + (4n-1) Y_0. \] 
Let $g \colon \GZ \rightarrow \Z$ be the homomorphism defined by $g(X_k) = 1 + |k|$ and $g(Y_k) = -1 - |k|$. 
Then \[ g \left( \HN(D_n) \right) = 2n^2 + 2n -1 .\] 
Let $R$ be the set of $\pm(X_k + Y_k)$, $\pm(X_k + Y_{k+1})$, 
$\pm(X_{k+1} - X_k)$ and $\pm(Y_{k+1} - Y_k)$,
for all integers $k$,
which represent all possible changes of $\HN$ under framed Reidemeister moves
(see \cite{HN08} for discussion on changes of $\HN$ under Reidemeister moves).
Since $|g(r)| \le 1$ for all $r \in R$, 
the lower bound for the number of framed Reidemeister moves 
to unknot $D_n$ obtained from $\HN$ is $2n^2 + 2n -1$.

This lower bound is higher than that found by $\SCI$, which has $\frac{3}{2}$ as the quadratic coefficient.
However, $\SCI$ is still useful, as it provides bounds on the minimal number of $\Omega3$ moves.
$g(\HN)$ does not provide such a bound, 
as it changes under unmatched $\Omega2$ moves. 
In fact, the change of $\HN$ under a~$\Omega3$ move 
can be expressed as a~sum of changes under a~matched and unmatched $\Omega2$ move.
Therefore, the value of $\HN$ is not sufficient to distinguish $\Omega3$ moves from
combinations of $\Omega2$ moves.

Finally, we show that that the minimal number of framed Reidemeister moves for unknotting 
can be degrees higher than the number of regular Reidemeister moves needed.
One example is the following family of unknot diagrams $L_n$.

\begin{figure}[ht]
    \centering
    \includegraphics[width=0.5\textwidth]{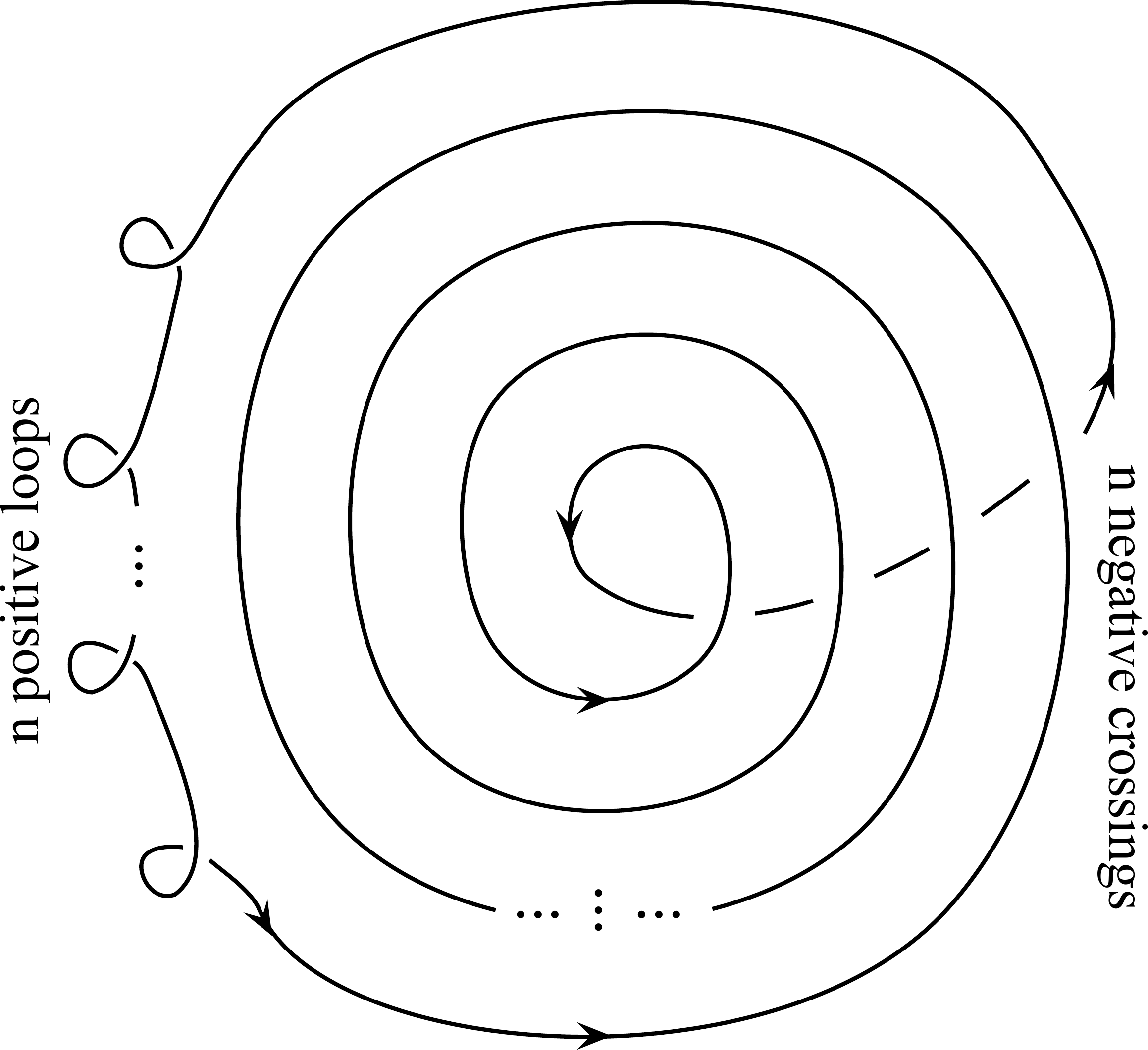}
    \caption{A family of unknots $L_n$.}
    \label{fig:loop-family}
\end{figure}

\begin{proof}[Proof of Theorem \ref{thm:framed-v-normal-Rmoves-comparison}]
    Clearly, $L_n$ may be unknotted using $2n$ Reidemeister moves.
    Moreover, $L_n$ has $2n$ crossings, and thus needs at least $n$ Reidemeister moves to unknot.
    Therefore the minimal unknotting sequence has length $\Theta(n)$.

    As a~framed unknot, $L_n$ may be unknotted inductively in the following way.
    ``Push'' a~loop from outside onto the loop in the middle of the diagram
    using $2n$ moves of type $\Omega2$ and $n$ moves of type $\Omega3$.
    Then, use a~$\Omega2$ move in the middle of the diagram to obtain $L_{n-1}$.

    On the other hand, calculating $\SCI$ for $L_n$ gives that
    \[ \SCI(L_n) = \frac{n(n+1)}{2}, \]
    so we need at least this number of $\Omega3$ moves to unknot $L_n$
    as a~framed unknot.
    This proves $L_n$ is optimally unknotted in $\Theta(n^2)$ moves.
\end{proof}

\begin{remark}
    Since $\SCI$ gives bounds for the number of $\Omega3$ moves,
    we may strengthen the bound in the proof above by considering
    other invariants (e.g. the number of crossings)
    and bounds on $\Omega1$ and $\Omega2$ moves that these provide.
\end{remark}

\section{Comparison with the Hass-Nowik invariant}

In this section, we recall the definition of a~knot diagram invariant $\HN$
given by Hass and Nowik in \cite{HN08}.
The invariant was used by Hass and Nowik to prove quadratic bounds for unknotting
a~family of diagrams almost identical to the family $D_n$.
We prove it is additive under connected sum
and is not a~Vassiliev invariant.
We end this section with a discussion of the relationship
between forward/backward and positive/negative character of $\Omega3$ moves,
which is established in Proposition \ref{thm:R3-classification-relation}.

\subsection{$\HN$ and its properties}
\begin{definition}[Hass-Nowik diagram invariant \cite{HN08}] \label{def:Hass-Nowik}
    Let $D$ be an~oriented knot diagram.
    Denote by $\lk$ the linking number of a~two-component link.
    For such $D$, we define
    \[ \HN(D) = \sum_{c \in \mathcal{C_+}(D)} X_{\lk(D^c)} 
    + \sum_{c \in \mathcal{C_-}(D)} Y_{\lk(D^c)} \]
    where $\mathcal{C_+}(D)$ be the set of positive crossings and $\mathcal{C_-}(D)$ 
    be the set of negative crossings of $D$, and $D^c$
    denotes the two-component link obtained by smoothing $D$ at $c$.
    This invariant takes values in $\G_\Z$,
    the free abelian group with basis $\{ X_k , Y_k \}_{s \in \Z}$. 
    \qet
\end{definition}

\begin{remark}
    In \cite{HN08}, $\HN$ is a~part of a~larger family
    $\mr{I}_\phi$ defined for any 2-component link invariant $\phi$.
    Precisely, $\HN = \mr{I}_{\lk}$.
\end{remark}

It turns out that $\HN$ is additive under connected sum,
similarly to $\SCI$.

\begin{theorem} \label{thm:HN-connected-sum}
    For any two knot diagrams $D$ and $E$, \[ \HN(D \# E) = \HN(D) + \HN(E) .\]
    \begin{proof} 
        Let $\mathcal{C}(D)$ and $\mathcal{C}(E)$ be the sets of crossings of $D \# E$ 
        that come from $D$ and $E$ respectively. 
        Let $\tilde D$ and $\tilde E$ be the parts of the diagram $D \# E$ 
        that come from $D$ and $E$, respectively. 
        The linking number of a~two-component link is equal to 
        the half of the sum of signs of crossings between the components.
        After smoothing a~crossing $ a \in \mathcal{C}(D)$, 
        $\tilde E$ is contained entirely within one of the two components. 
        Thus, none of the crossings in $\mathcal{C}(E)$ contribute to the linking number 
        of the two-component link, 
        meaning that the link $(D \# E)^a$ has the same linking number as $D^a$.

        The same reasoning shows that $\lk(D \# E)^b = \lk(E^b)$ for any $b \in \mathcal{C}(E)$.
        Thus, $\HN(D \# E) = \HN(D) + \HN(E)$.
    \end{proof}
\end{theorem}

Unlike $\SCI$, $\HN$ is not a~Vassiliev invariant.
We show this using the standard diagrams of $(2,p)$-torus knots
(for $p$ odd), which we denote $T(2,p)$.
These diagrams are characterized by the property that 
they have $p$ positive crossings and are alternating 
(cf. Figure \ref{fig:torus-2+5}).

\begin{figure}[ht]
    \centering
    \includegraphics[width=0.2\textwidth]{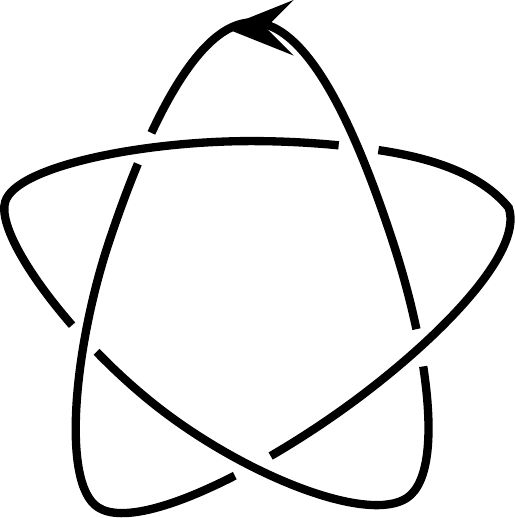}
    \caption{Diagram of the $(2,5)$-torus knot}
    \label{fig:torus-2+5}
\end{figure}

\begin{lemma} \label{lem:torus-knot-HN}
    Let $S$ be the set of all crossings of $T(2,p)$. 
    Let $S_k$ be any subset of $S$ with cardinality $k$, 
    and $T(2,p)_{S_k}$ be the knot diagram of $T(2,p)$ with the crossings of $S_k$ changed. 
    Then
    \begin{equation*}
        \HN(T(2,p)_{S_k}) = (p-k)X_{\frac{p-2k-1}{2}} + kY_{\frac{p-2k+1}{2}} .
    \end{equation*}
    \begin{proof}
        First, notice that any crossing created by smoothing
        a~crossing in $T(2,p)$ is a~crossing between different components
        of the link obtained.
        Since $T(2,p)$ differs from $T(2,p)_{S_k}$ by just crossing changes,
        the same is true for $T(2,p)_{S_k}$.

        Therefore, smoothing a~positive crossing of $T(2,p)_{S_k}$
        leaves $p-k-1$ positive crossings and $k$ negative crossings,
        so each positive crossing contributes $X_{\frac{p-2k-1}{2}}$ to $\HN$.
        Similarly, smoothing a~negative crossing contributes
        $Y_{\frac{p-2k+1}{2}}$ to $\HN$ since there are $p-k$ positive
        and $k-1$ negative crossings left.
    \end{proof}
\end{lemma}

\begin{theorem}\label{thm:Hass-Nowik-finite-type}
    $\HN$ is not a Vassiliev diagram invariant.
    \begin{proof}
        Let $\mc{C}_p$ be the set of all crossings of $T(2,p)$.
        From Lemma \ref{lem:torus-knot-HN} we obtain that
        \begin{equation*}
            \HN_{\mc{C}_p}(T(2,p)) = 
            \sum_{S \subset \mc{C}_p}
            (p-|S|) X_{\frac{p-2|S|-1}{2}} + |S| Y_{\frac{p-2|S|+1}{2}}
        \end{equation*}
        and since there is only one subset $S \subset \mc{C}_p$ such that $|S|=0$,
        therefore the coefficient of $X_{\frac{p-1}{2}}$ in the sum above
        is equal to $p \neq 0$,
        so $\HN_{\mc{C}_p}(T(2,p)) \neq 0$.
        This finishes the proof,
        since if $\HN$ was a~Vassiliev invariant of order $n$,
        then for any $p>n$ we would have $\HN_{\mc{C}_p}(T(2,p)) = 0$.
    \end{proof}
\end{theorem}


\subsection{Forward/backward, positive/negative and ascending/descending $\Omega3$ moves}
\label{subs:type-3-moves}
The relationship of forward/backward to positive/negative $\Omega3$ moves
is best understood using the notions of \emph{ascending} and \emph{descending}
$\Omega3$ moves introduced by \"Ostlund \cite{O01}:

\begin{definition} [ascending and descending $\Omega3$ moves]
    Follow the orientation of the knot diagram. 
    An $\Omega3$ move is \emph{ascending} if the three segments involved 
    are passed in the order bottom-middle-top, 
    and \emph{descending} if the three segments involved 
    are passed in the order top-middle-bottom.
    \qet
\end{definition}

\begin{remark}
    The ascending/descending classification of a move 
    does not change when reversing the move.
    For instance, if we consider diagrams in Figure \ref{fig:R3aAD},
    it does not matter if we go \emph{from left to right}
    or \emph{from right to left}.
    On the contrary, reversing a~move changes its forward/backward
    or positive/negative classification.
\end{remark}

\begin{proposition}\label{thm:R3-classification-relation}
    An~ascending $\Omega3$ move is forward if and only if it is positive.
    A~descending $\Omega3$ move is forward if and only if it is negative.
    \begin{proof}
        The eight cases of $\Omega3$ moves can be placed into two groups 
        based on the bottom-middle-top orientation of the vanishing triangle. 
        In cases $a, d, e,$ and $g$ it is clockwise, 
        and in cases $b, c, f,$ and $h$ it is counterclockwise. 
        We will prove Proposition \ref{thm:R3-classification-relation} for 
        $\Omega3a$ and $\Omega3b$ moves, 
        as others are dealt with similarly.

        \begin{figure}[ht]
            \centering
            \includegraphics[width=1\textwidth]{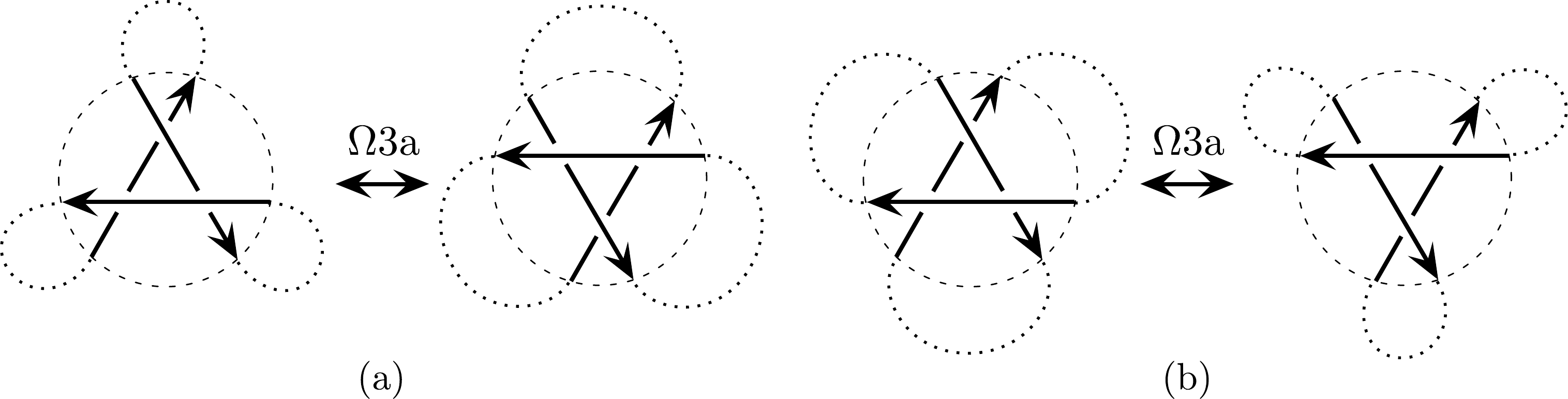}
            \caption{Diagrams for forward (a) ascending and (b) descending $\Omega3a$ moves.}
            \label{fig:R3aAD}
        \end{figure}

        For an~ascending $\Omega3a$ move, of the three strands involved in the move, 
        the bottom strand connects to the middle strand, 
        the middle to the top, and the top to the bottom (see Figure \ref{fig:R3aAD}).
        The order-of-appearance orientation of the vanishing triangle, 
        introduced in Definition \ref{def:R3-positive}, is clockwise,
        which agrees with orientations of all three strands.
        This makes $q' = -1$ for this diagram. 
        Once the $\Omega3$ move is made, 
        the order-of-appearance orientation remains clockwise, 
        but it now disagrees with orientations of all three strands, so $q' = +1$. 
        It follows that a forward ascending $\Omega3a$ move is positive.

        For a~descending $\Omega3a$ moves, 
        of the three strands involved in the move, 
        the bottom strand connects to the top strand, 
        the top to the middle, and the middle to the bottom (see Figure \ref{fig:R3aAD}).
        The order-of-appearance orientation of the vanishing triangle is counterclockwise. 
        This makes $q' = +1$ for this diagram. 
        Once the $\Omega3$ move is made, $q'$ becomes $-1$. 
        Thus, a forward descending $\Omega3a$ move is negative.

        \begin{figure}[ht]
            \centering
            \includegraphics[width=1\textwidth]{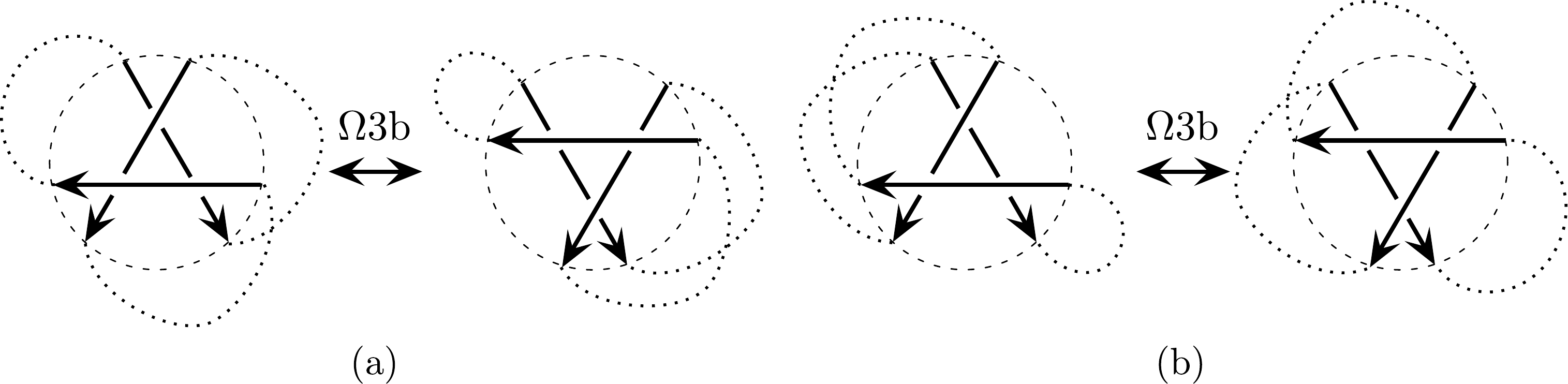}
            \caption{Diagrams for forward (a) ascending and (b) descending $\Omega3b$ moves.}
            \label{fig:R3bAD}
        \end{figure}

        For an~ascending $\Omega3b$ move, of the three strands involved in the move, 
        the bottom strand connects to the middle strand, the middle to the top, 
        and the top to the bottom (see Figure \ref{fig:R3bAD}).
        The order-of-appearance orientation of the vanishing triangle is counterclockwise. 
        This makes $q' = -1$ for this diagram. 
        Once the $\Omega3$ move is made, $q'$ becomes $+1$. 
        Thus, a forward ascending $\Omega3b$ move is positive.

        For a~descending $\Omega3b$ move, the order-of-appearance orientation is opposite
        (as it was in the case of $\Omega3a$ move),
        so a forward descending $\Omega3b$ move is negative.
    \end{proof}
\end{proposition}

With the relationships between different kinds of $\Omega3$ moves sorted out,
we can precisely describe the behavior of $\HN$ under Reidemeister moves:
\begin{proposition}\label{prop:HN-changes}
    Under forward Reidemeister moves, 
    $\HN$ changes by:
    \begin{itemize}
        \item $X_0$ (resp. $Y_0$) under $\Omega1$ moves
            creating positive (resp. negative) crossing;
        \item $X_k + Y_{k+1}$ (for some $k \in \Z$) under $\Omega2m$ moves;
        \item $X_k + Y_k$ under $\Omega2u$ moves;
        \item $X_k - X_{k+1}$ (resp. $Y_{k+1} - Y_k$) under ascending (i.e. positive) 
            $\Omega3$ moves with positive (resp. negative) crossing between top and bottom strands;
        \item $X_{k+1} - X_k$ (resp. $Y_k - Y_{k+1}$) under descending (i.e. negative) 
            $\Omega3$ moves with positive (resp. negative) crossing between top and bottom strands;
    \end{itemize}
    \begin{proof}
        Changes of $\HN$ under Reidemeister moves are described section $2$ of \cite{HN08}.
        The only detail not determined there is the sign of changes under moves of type $\Omega3$,
        i.e. whether the change is equal to $X_k-X_{k+1}$ or $-X_k+X_{k+1}$
        (resp. $Y_k-Y_{k+1}$ or $-Y_k+Y_{k+1}$).
        This can be easily checked through casework 
        as demonstrated in Proposition 3.5 in \cite{S17} for ascending $\Omega3a$ moves.
    \end{proof}
\end{proposition}

From $\HN$ we may also obtain \emph{cowrithe}. 
While $\SCI$ seems to closely resemble $\St$, 
the changes of $\SCI$ under $\Omega3$ moves depend on the forward/backward character
of the $\Omega3$ move,
and the changes of $\St$ depend on the positive/negative quality of the move.
Cowrithe resembles the behavior of $\St$ in that its change
under $\Omega3$ moves depends only the positive/negative type of the move
(since it is the case with $\HN$).

\begin{definition}[cowrithe, \cites{H06,HN08}] \label{def:cowrithe}
    Let $D$ be an oriented knot diagram. 
    Let $f\colon \GZ \to \Z$ be the homomorphism 
    defined by $f(X_n) = -n$ and $f(Y_n) = n$. 
    Then the cowrithe of $D$ is $f(\HN_{\lk})$.
\end{definition}

From Proposition \ref{prop:HN-changes} we obtain a~description of changes of cowrithe.
\begin{corollary} \label{thm:cowrithe-changes}
    Cowrithe does not change under $\Omega1$ and unmatched $\Omega2$ moves.
    It increases by $1$ under forward matched $\Omega2$ and positive $\Omega3$ moves.
\end{corollary}

\newpage
\section{Appendix}
In Table \ref{tab:invariants} we summarize how some of the known knot diagram invariants
change under various types of \emph{forward} Reidemeister moves.

\begin{table}[ht]
\centering
\noindent\makebox[\textwidth]{
\begin{tabular}{|c|c|c|c|c|c|c|}
	\hline
    Invariant & $\Omega1$ & $\Omega1\mr{F}$ & $\Omega2m$ & $\Omega2u$ & $\Omega3^{asc}$ & $\Omega3^{desc}$ \\
	\hline
	$\writhe=$writhe & $\sgn(c) = \pm 1$ & 0 & 0 & 0 & 0 & 0 \\
	\hline
    $n=$number of crossings & +1 & +2 & +2 & +2 & 0 & 0 \\
	\hline
    winding number & $w(c) = \pm 1$ & $2 w(c)$ & 0 & 0 & 0 & 0 \\
	\hline
    SCI & $\sgn(c)\ind(c)$ & 0 & 0 & 0 & +1 & +1 \\
	\hline
    $\HN$ & $X_0$ if $\sgn(c) = +1$, & $X_0+Y_0$ & $X_k+Y_{k+1}$ & $X_k+Y_k$ & $X_k-X_{k+1}$, & $X_{k+1}-X_k$, \\
      & $Y_0$ if $\sgn(c) = -1$ & & & & $Y_{k+1}-Y_k$ & $Y_k-Y_{k+1}$ \\
	\hline
    $x=$cowrithe & 0 & 0 & +1 & 0 & +1 & -1 \\
	\hline
    
    $\St$ & $w(c) \ind(c)$ & $2 w(c) \ind(c)$ & 0 & 0 & +1 & -1 \\
    \hline
    $\Jp$ & $-2w(c)\ind(c)$ & $-4w(c)\ind(c)$ & +2 & 0 & 0 & 0 \\
	\hline
    ${\Jp}/2 + \St$ & 0 & 0 & +1 & 0 & +1 & -1 \\
	\hline
    $A_n$ & 0 & 0 & 0 & 0 & ? & 0 \\
    \hline
    $D_n$ & 0 & 0 & 0 & 0 & 0 & ? \\
    \hline
    $W_n$ & 0 & 0 & 0 & 0 & ? & ? \\
    \hline
    
\end{tabular}}
\caption{Knot diagram invariants and their changes under 
    forward Reidemeister moves for knots and framed knots.
$c$ denotes a~crossing created by the $\Omega1$ (or $\Omega1\mr{F}$) move,
and "?" denotes lack of a~combinatorial formula for the change.}
\label{tab:invariants}
\end{table}

Note that a~crossing $c$ created by a~move of type $\Omega1$ (or $\Omega1\mr{F}$)
has $w(c) = 1$ if it is created on the left side of the strand
and $w(c) = -1$ if it is created on the right side of it
($w(c)$ is the \emph{weight} of the crossing, see Definition \ref{def:weight}).

As observed by Hass and Nowik \cite{HN08},
the difference between cowrithe and $\Jp/2 + \St$ is equal to
$4 c_2$, where $c_2$ is the second coefficient of the Conway polynomial of a~knot.
The fact that this difference is a~knot invariant
is reflected in Table \ref{tab:invariants}.

The precise description of the entries for $\HN$ in the table above
belongs to Proposition \ref{prop:HN-changes} and mostly follows \cite{HN08},
and from these one computes the changes for cowrithe.
Results for $\St$ and $\Jp$ follow directly from Shumakovich's
(Theorem \ref{thm:explicit-St}) and Viro's (Theorem \ref{thm:explicit-Jpm}) formulas.
The invariants $A_n, D_n$ for $n = 4,5,6, \ldots$ and $W_n$ for $n=3,5,7,\ldots$
are described in \cite{O01}.
It is worth to note that there is a~combinatorial formula for the change of $W_3$
under moves of type $\Omega3$ (see \cite{O01}). 
While the change may be large, 
it is bounded by the number of the crossings of a~diagram.


\newpage
\begin{bibdiv}
    \begin{biblist}*{labels={numeric}}
        \bib{Arn94}{article}{
            author={Arnold, Vladimir I.},
            title={Plane curves, their invariants,
            perestroikas and classifications},
            conference={
                title={Singularities and bifurcations},
            },
            book={
                series={Adv. Soviet Math.},
                volume={21},
                publisher={Amer. Math. Soc.},
                address={Providence, RI},
            },
            date={1994},
            pages={33--91},
        } 
        \bib{CDM12}{book}{
            author={Chmutov, Sergei},
            author={Duzhin, Sergei},
            author={Mostovoy, Jacob},
            title={Introduction to Vassiliev knot invariants},
            publisher={Cambridge University Press, Cambridge},
            date={2012},
        }
        \bib{H06}{article}{
            author={Hayashi, Chuichiro},
            title={A lower bound for the number of Reidemeister moves for unknotting},
            journal={J. Knot Theory Ramifications},
            volume={15},
            date={2006},
            number={3},
            pages={313--325},
        }
        \bib{HHSY}{article}{
            author={Hayashi, Chuichiro},
            author={Hayashi, Miwa},
            author={Sawada, Minori},
            author={Yamada, Sayaka},
            title={Minimal unknotting sequences of Reidemeister moves containing unmatched RII moves},
            journal={J. Knot Theory Ramifications},
            volume={21},
            date={2012},
            number={10},
            pages={1250099, 13 pages},
        }
        \bib{HN08}{article}{
            author={Hass, Joel},
            author={Nowik, Tahl},
            title={Invariants of knot diagrams},
            journal={Math. Ann.},
            volume={342},
            number={1},
            date={2008},
            pages={125--137},
        }
        \bib{HN10}{article}{
            author={Hass, Joel},
            author={Nowik, Tahl},
            title={Unknot diagrams requiring a quadratic number of Reidemeister moves to untangle},
            journal={Discrete Comput. Geom.},
            volume={44},
            number={1},
            date={2010},
            pages={91--95},
        }
        \bib{L15}{article}{
            author={Lackenby, Marc},
            title={A polynomial upper bound on Reidemeister moves},
            journal={Ann. Math.},
            volume={182},
            date={2015},
            number={2},
            pages={491--564},
        }
        \bib{O01}{article}{
            author={{\"O}stlund, Olof-Petter},
            title={Invariants of knot diagrams and relations among Reidemeister moves},
            journal={J. Knot Theory Ramifications},
            volume={10},
            number={8},
            date={2001},
            pages={1215--1227},
        }
        \bib{Sh95}{article}{
            author={Shumakovich, Alexander},
            title={Explicit formulas for the strangeness of plane curves},
            language={Russian, with Russian summary},
            journal={Algebra i Analiz},
            volume={7},
            date={1995},
            number={3},
            pages={165--199},
            translation={
                journal={St. Petersburg Math. J.},
                volume={7},
                date={1996},
                number={3},
                pages={445-472},
            },
        }
        \bib{S17}{article}{
            author={Suwara, Piotr},
            title={Minimal generating sets of directed oriented Reidemeister moves},
            journal={J. Knot Theory Ramifications},
            volume={26},
            date={2017},
            number={4},
            pages={1750016, 20},
        }
        \bib{V94}{article}{
            author={Vassiliev, Victor A.},
            title={Invariants of ornaments},
            conference={
                title={Singularities and bifurcations},
            },
            book={
                series={Adv. Soviet Math.},
                volume={21},
                publisher={Amer. Math. Soc.},
                address={Providence, RI},
            },
            date={1994},
            pages={225--262},
        }
        \bib{Vi94}{book}{
            author={Viro, Oleg},
            title={First degree invariants of generic curves on surfaces},
            series={UUDM report},
            publisher={Department of Mathematics, Uppsala University},
            date={1994},
        }
    \end{biblist}
\end{bibdiv}
\end{document}